\documentclass[11pt]{amsproc}%
\usepackage{amssymb}
\usepackage{amsfonts}
\usepackage{amsmath}
\usepackage{graphicx}%
\providecommand{\U}[1]{\protect\rule{.1in}{.1in}}
\theoremstyle{plain}

\newtheorem{lemma}{Lemma}

\newtheorem{theorem}{Theorem}
\numberwithin{equation}{section}
\begin{document}
\title[Approximation by Lipschitz, analytic maps]{Approximation by Lipschitz, analytic maps on certain Banach spaces}
\author{R. Fry}
\address{Department of Mathematics and Statistics\\
Thompson Rivers University\\
Kamloops, BC\\
Canada}
\email{rfry@tru.ca}
\urladdr{}
\author{L. Keener}
\curraddr{Department of Mathematics\\
University of Northern British Columbia\\
Prince George, BC\\
Canada}
\email{keener@unbc.ca}
\urladdr{}
\thanks{}
\dedicatory{ }
\begin{abstract}
We show that on separable Banach spaces admitting a separating polynomial, any
uniformly continuous, bounded, real-valued function can be uniformly
approximated by Lipschitz, analytic maps on bounded sets.

\end{abstract}
\maketitle

\section{Introduction}

\medskip

\subsection{History of the Problem}

Investigating the uniform approximation of continuous functions by smooth
functions has a long history (`function' in this paper shall mean real-valued
function). For continuous functions $f$ defined on closed intervals in
$\mathbb{R}$, Weierstrass's well known classical theorem asserts that $f$ can
be uniformly approximated by polynomials. For continuous functions $f$ on open
subsets of $\mathbb{R}^{n},$ H. Whitney showed in the famous work \cite{W}
that $f$ can be uniformly approximated by (real) analytic maps.

For infinite dimensional (real) Banach spaces, the situation is far more
difficult. In fact, it was proven in \cite{NS} that even on the closed unit
ball of separable Hilbert space, $C^{\infty}$-smooth functions cannot
generally be approximated by polynomials (see below for the relevant
definitions). Nevertheless, for certain Banach spaces Whitney's result can be
extended to infinite dimensions. A classical and often cited theorem of
Kurzweil \cite{K} states in particular that if $X$ is a separable Banach space
admitting a separating polynomial, then any continuous function on $X$ can be
uniformly approximated by analytic maps on $X.$ Examples of such $X$ include
$l_{p}$ or $L_{p},$ with $p$ an even integer.

Most subsequent work concerning the uniform approximation of continuous
functions on Banach spaces focused on approximation by (merely) $C^{p}$-smooth
functions rather than analytic maps (see e.g., \cite{DGZ}, \cite{FHHMPZ}), the
primary reason for this being the ability to employ $C^{p}$-smooth partitions
of unity in this context, which cannot be used in the analytic case. As a
consequence, the number of results on analytic approximation in this area are few.

However, over the last ten years two important papers have appeared. In
\cite{DFH1} it is shown that in $l_{p}$ or $L_{p},$ with $p$ an even integer,
that any equivalent norm can be uniformly approximated on bounded sets by
analytic norms. In \cite{DFH2} it is proven for example, that in $X=c_{0}$ or
$X=C\left(  K\right)  ,$ with $K$ a countable compact, that any equivalent
norm can be uniformly approximated by analytic norms on $X\backslash\left\{
0\right\}  .$

An independent line of investigation concerning approximation by Lipschitz,
$C^{p}$-smooth functions was recently undertaken in a series of papers
including \cite{F1}, \cite{F2}, \cite{AFM}, and \cite{F3}. For example, in
\cite{AFM} it is shown in particular that on separable Banach spaces admitting
Lipschitz, $C^{p}$-smooth bump functions, that any bounded, uniformly
continuous function can be uniformly approximated by Lipschitz, $C^{p}$-smooth
functions. This result was then generalized to weakly compactly generated
Banach spaces in \cite{F3}. In terms of smoothness class ($C^{p},C^{\omega},$
etc.), the present paper in some sense completes this programme. These results
have found applications in the area of deleting diffeomorphisms on Banach
spaces \cite{AM}, and in variational principles on Riemannian manifolds
\cite{AFL}. For superreflexive spaces, we note that the approximation of
Lipschitz functions by Lipschitz, $C^{1}$-smooth maps can be accomplished via
convolution techniques (see, e.g., \cite{C}, \cite{LL}).

The motivation for this article was to see if the recent results on Lipschitz,
smooth approximation could be achieved in the context of analytic
approximation in the spirit of Kurzweil's work mentioned above. We give a
positive solution to this problem for bounded and uniformly continuous
functions on bounded sets in separable spaces admitting a separating
polynomial. We remark that the uniform continuity is a necessary condition
here. To our knowledge, this is the only result on approximation (of general
functions as opposed to norms) by Lipschitz, analytic functions on infinite
dimensional spaces.

\medskip

Specifically, we establish,

\begin{theorem}
\label{Maintheorem}Let $X=\left(  X,\left\Vert \cdot\right\Vert _{X}\right)  $
be a separable, real Banach space that admits a separating polynomial, let
$G\subset X$ be a bounded open set, and let $F:G\rightarrow\mathbb{R}$ be
bounded and uniformly continuous. Then for each $\varepsilon>0$ there is a
real analytic function $K:G\rightarrow\mathbb{R}$ which is Lipschitz on $G$
and such that $\left\vert F(x)-K(x)\right\vert <\varepsilon$ for all $x\in G$.
\end{theorem}

\subsection{Basic Definitions}

Our notation is standard, with $X$ denoting a Banach space, and an open ball
with centre $x$ and radius $r$ denoted $B_{r}\left(  x\right)  ,$ and the
boundary of the unit ball in the space $X$ is denoted by $S_{X}$. Function
shall always mean real-valued function. If $\left\{  f_{j}\right\}  _{j}$ is a
sequence of Lipschitz functions on $X,$ then we will at times say this family
is \textit{uniform Lipschitz (UL)} if there is a common Lipschitz constant for
all $j.$ A \textit{homogeneous polynomial of degree }$n$ is a map,\textit{\ }%
$P:X\rightarrow\mathbb{R},$ of the form $P\left(  x\right)  =A\left(
x,x,...,x\right)  ,$ where $A:X^{n}\rightarrow\mathbb{R}$ is $n-$multilinear
and continuous. For $n=0$ we take $P$ to be constant. A \textit{polynomial of
degree }$n$ is a sum $\sum_{i=0}^{n}P_{i}\left(  x\right)  ,$ where the
$P_{i}$ are $i$-homogeneous polynomials.

Let $X$ be a Banach space, and $G\subset X$ an open subset. A function
$f:G\rightarrow\mathbb{R}$ is called \textit{analytic }if for every $x\in G,$
there are a neighbourhood $N_{x},$ and homogeneous polynomials $P_{n}%
^{x}:X\rightarrow\mathbb{R}$ of degree $n$, such that
\[
f\left(  x+h\right)  =\sum_{n\geq0}P_{n}^{x}\left(  h\right)
\;\text{provided\ }x+h\in N_{x}.
\]

Further information on polynomials may be found, for example, in $\left[
\text{SS}\right]  .$

For a Banach space $X,$ we define its \textit{complexification }$\widetilde
{X}$ in the standard way. That is, $\widetilde{X}=X\bigoplus iX$ with norm
\[
\left\|  z\right\|  _{\widetilde{X}}=\left\|  x+iy\right\|  _{\widetilde{X}%
}=\sup_{0\leq\theta\leq2\pi}\left\|  \cos\theta\ x-\sin\theta\ y\right\|
_{X}.
\]
If $q\left(  x\right)  $ is a polynomial on $X,$ there is a natural extension
of $q\left(  x\right)  $ to a polynomial $\widetilde{q}\left(  z\right)
=\widetilde{q}\left(  x+iy\right)  $ on $\widetilde{X}$ where for $y=0$ we
have $\widetilde{q}=q.$ For more information on complexification (and
polynomials) we recommend \cite{MST}.

We define $\widetilde{c}_{0}=\left\{  \left\{  z_{j}\right\}  :z_{j}%
\in\mathbb{C},\ \left|  z_{j}\right|  \rightarrow0\right\}  ,$ with norm
$\left\|  z\right\|  _{\widetilde{c}_{0}}=\left\|  \left\{  z_{j}\right\}
\right\|  _{\widetilde{c}_{0}}=\max_{j}\left\{  \left|  z_{j}\right|
\right\}  ,$ and a similar definition for $\widetilde{l}_{\infty}.$ In the
sequel, all extensions of functions from $X$ to $\widetilde{X},$ as well as
subsets of $\widetilde{X},$ will be embellished with a tilde.

\medskip

The proof of Theorem 1 is broken up into several sections and lemmas which we
now present.

\section{Preliminary Results}

\subsection{An extension of the Preiss norm}

As developed in \cite{FPWZ}, there is a an analytic norm on $c_{0}$ (hereafter
referred to as the Preiss norm) that is equivalent to the canonical supremum
norm. Let us recall the construction. We may define this equivalent norm
$\left\Vert \cdot\right\Vert $ as follows. Let $C:c_{0}\rightarrow\mathbb{R}$
be given by $C(\{x_{n}\})=\sum_{n=1}^{\infty}\left(  x_{n}\right)  ^{2n}$. Let
$W=\{x\in c_{0}:C(x)\leq1\}.$ Then $\left\Vert \cdot\right\Vert $ is the
Minkowski functional of $W$; that is, $\left\Vert x\right\Vert $ is the
solution for $\lambda$ to $C\left(  \lambda^{-1}x\right)  =1.$ The Preiss norm
is analytic at all non-zero points in $c_{0}.$ To see this, let us define the
function $\widetilde{C}:V\rightarrow\mathbb{C}$ by $\widetilde{C}\left(
\left\{  z_{n}\right\}  \right)  =\sum_{n=1}^{\infty}\left(  z_{n}\right)
^{2n}$ where $V$ is the subset of $\widetilde{l}_{\infty}$ for which the
series converges. Then $\widetilde{C}$ is analytic at each $z\in\widetilde
{c}_{0}$. Indeed, the partial sums are analytic as a consequence of the
analyticity of the projection functions $p_{j}(\{z_{i}\})=z_{j}$, whose local
differentiability is easily shown by a direct calculation. Since the series in
the definition of $\widetilde{C}$ converges locally uniformly at each
$z\in\widetilde{c}_{0}$ the analyticity of $\widetilde{C}$ on $\widetilde
{c}_{0}$ follows. Also, for $z\in\widetilde{c}_{0}$ sufficiently close to
$c_{0}$ we have for $\lambda\in\mathbb{C}\backslash\left\{  0\right\}  ,$
$\frac{\partial\widetilde{C}\left(  \lambda^{-1}z\right)  }{\partial\lambda
}\neq0,$ hence one can apply the complex Implicit Function Theorem (see e.g.,
\cite{D} page 265, where the real result for Banach spaces is easily extended
to the analytic case) to $F\left(  z,\lambda\right)  =\widetilde{C}\left(
\lambda^{-1}z\right)  -1$ to obtain a unique analytic solution $\lambda\left(
z\right)  $ to $F\left(  z,\lambda\right)  =0,$ with $\lambda\mid_{c_{0}%
}=\left\Vert \cdot\right\Vert .$ Now if $x=\{x_{n}\}$ satisfies $\left\Vert
x\right\Vert _{c_{0}}=1,$ then $\sum_{n=1}^{\infty}\left(  \left\Vert
x\right\Vert ^{-1}x_{n}\right)  ^{2n}=1$ implies $\left\Vert x\right\Vert
\geq1.$ On the other hand, if $\left\Vert x\right\Vert _{c_{0}}=1/2,$ then
$C\left(  x\right)  \leq\sum_{n=1}^{\infty}(1/2)^{2j}<1,$ implying $\left\Vert
x\right\Vert <1$. Hence, $\left(  1/2\right)  \left\Vert x\right\Vert
\leq\left\Vert x\right\Vert _{c_{0}}\leq\left\Vert x\right\Vert $ for all $x$
in $c_{0}$.We shall use the above notation throughout this article. We now
extend the Preiss norm to an analytic function on a open set containing
$c_{0}$ in $l_{\infty}$ as follows. First define the set%

\[
U=\left\{
\begin{array}
[c]{c}%
\left\{  x_{j}\right\}  \in l_{\infty}:\text{there exists }j_{0}\text{ and
}0<a<3/4\ \\
\text{ \ \ \ \ \ \ \ with \ }j>j_{0}\Rightarrow\left|  x_{j}\right|  <a
\end{array}
\right\}  .
\]
Then $U$ is convex and open in $l_{\infty}$ and $c_{0}\subset U\subset
l_{\infty}.$

\begin{lemma}
There is an analytic function $\lambda$ defined on $U$ such that for $x\in
c_{0}\subset U,$ we have $\lambda\left(  x\right)  =\left\Vert x\right\Vert ,$
the latter being the Preiss norm.
\end{lemma}

\begin{proof}
Consider the function $C$ from above extended to $U$. By definition of $U,$ if
$x=\left\{  x_{j}\right\}  \in U,$ there is a neighbourhood $\widetilde{N}%
_{0}\subset\widetilde{l}_{\infty}$ of $x$ and a $j_{0}$ such that $y=\left\{
y_{j}\right\}  \in\widetilde{N}_{0}$ and $j>j_{0}$ imply $\left\vert
y_{j}\right\vert <3/4$. Also, it is convenient to introduce a neighbourhood
$\widetilde{N}_{1}\supset$ $\widetilde{N}_{0}$ and a $j_{1}$ such that
$y=\left\{  y_{j}\right\}  \in\widetilde{N}_{1}$ and $j>j_{1}$ imply
$\left\vert y_{j}\right\vert <7/8$. Observe that $\widetilde{C}\left(
\left\{  y_{j}\right\}  \right)  $ converges uniformly on $\widetilde{N}_{1}.$
So $\widetilde{C}\left(  \left\{  z_{j}\right\}  \right)  $ has a complex
derivative on $\widetilde{N}_{1},$ and it follows that $C$ is (real) analytic
on $U$

Now define $F\left(  \left\{  z_{j}\right\}  ,\lambda\right)  =F\left(
z,\lambda\right)  =\widetilde{C}\left(  \lambda^{-1}z\right)  -1$ as above,
here on $\widetilde{N}_{0}\times L,$ where $L\subset\mathbb{C}\backslash
\left\{  0\right\}  $ is such that $\lambda\in L\Rightarrow\left\vert
\lambda-1\right\vert \leq c,$ for some fixed $0<c<1/8.$ Then $F\left(
z,\lambda\right)  $ is (complex) analytic on $\widetilde{N}_{0}\times L$ since
$\lambda^{-1}z$ is in $\widetilde{N}_{1}$, and $\frac{\partial F}%
{\partial\lambda}\left(  z,\lambda\right)  =-\sum_{n=1}^{\infty}2n\left(
\lambda^{-1}z_{n}\right)  ^{2n}\lambda^{-1}.$ Clearly if $0\neq z=x\in
\widetilde{N}_{0}\cap U\subset l_{\infty},$ then this last expression is not
zero, and so by continuity of $\frac{\partial F}{\partial\lambda}\left(
z,\lambda\right)  $, choosing $\widetilde{N}_{0}$ and $c$ smaller if
necessary, we have $\frac{\partial F}{\partial\lambda}\left(  z,\lambda
\right)  \neq0.$ Hence we may apply the complex Implicit Function Theorem to
the equation $F\left(  z,\lambda\right)  =\widetilde{C}\left(  \lambda
^{-1}z\right)  -1$ to obtain a unique complex analytic solution $\lambda
\left(  z\right)  =\lambda\left(  \left\{  z_{j}\right\}  \right)  $ to
$\sum_{n=1}^{\infty}\left(  \lambda^{-1}z_{n}\right)  ^{2n}=1$ on
$\widetilde{N}_{0}\times L\subset\widetilde{l}_{\infty}\times\mathbb{C}%
\backslash\left\{  0\right\}  .$ It follows that $\lambda$ restricts to a
(real) analytic function on $U$. Easily $\lambda\mid_{c_{0}}=\left\Vert
\cdot\right\Vert .$
\end{proof}

\noindent The function $\lambda$ from the lemma, being the solution to
$\sum_{n=1}^{\infty}\left(  \lambda^{-1}x_{n}\right)  ^{2n}=1,$ is the
Minkowski functional of $\mathcal{S}$, where
\[
\mathcal{S}=\left\{  \left\{  x_{j}\right\}  \in U:C\left(  \left\{
x_{j}\right\}  \right)  =\sum_{n=1}^{\infty}\left(  x_{n}\right)  ^{2n}%
\leq1\right\}  \subset U.
\]
Observe that $0\in\ $interior$\left(  \mathcal{S}\right)  .$ Now because the
function $C\left(  \left\{  x_{j}\right\}  \right)  $ is convex on $U,$
$\mathcal{S}$ is convex, and it follows that if $x,y\in U$ with $x+y\in U,$
then $\lambda\left(  x+y\right)  \leq\lambda\left(  x\right)  +\lambda\left(
y\right)  .$ This shows that for such $x,y\in U,$ we have the 1-Lipschitz
property, $\left\vert \lambda\left(  x\right)  -\lambda\left(  y\right)
\right\vert \leq\lambda\left(  x-y\right)  .$ Also, $\mathcal{S}$ is balanced,
and hence for any $x\in U$ and $a\in\mathbb{R}$ with $ax\in U,$ we have in
this case that $\lambda\left(  ax\right)  =\left\vert a\right\vert
\lambda\left(  x\right)  .$ Next, if $x=\left\{  x_{j}\right\}  \in U,$ then
from the argument establishing bounds for the Preiss norm, we also have here
that $\lambda\left(  \left\{  x_{j}\right\}  \right)  \geq\left\Vert
x\right\Vert _{\infty}.$ By the homogeneity property, this last inequality
holds for any $x\in U$ and $a\in\mathbb{R}$ so that $ax\in U.$ Similarly, for
$x\in U$ we have as before that $\frac{1}{2}\lambda\left(  \left\{
x_{j}\right\}  \right)  \leq\left\Vert x\right\Vert _{\infty}$, with this
estimate holding for any $x\in U$ and $a\in\mathbb{R}$ so that $ax\in U.$ We
use these estimates in the sequel.

\subsection{Polynomials}

Let $X$ be a Banach space. A \textit{separating polynomial} on $X$ is a
polynomial $q$ on $X$ such that $0=q(0)<\inf\{|q(x)|:x\in S_{X}\}$. It is
known [FPWZ] that if $X$ is superreflexive and admits a $C^{\infty}$-smooth
bump function then $X$ admits a separating polynomial. The following lemma
makes precise, observations of Kurzweil in [K].

\begin{lemma}
\label{polynomial}Let $X$ be a real Banach space with norm $\left\|
\cdot\right\|  _{X}$ and suppose that there is a separating polynomial $p$ of
degree $n$ on $X$. Then there is a polynomial $q$ on $X$ such that $\left\|
y\right\|  _{X}^{2n}\leq q(y)$ for all $y$ in $X$ with $q(y)<1$. Furthermore,
there is a constant $K_{1}>0$ such that $q(y)\leq K_{1}\max\{\left\|
y\right\|  _{X},\left\|  y\right\|  _{X}^{2n}\}$ for all $y$ in $X$.
\end{lemma}

\begin{proof}
Let $p=p_{1}+p_{2}+\cdots+p_{n}$, where $p_{i}$ is $i$-homogeneous for $1\leq
i\leq n$. Define $q=\sum_{i=1}^{n}q_{i}$ where $q_{i}=p_{i}^{2}$ is $2i
$-homogeneous for $1\leq i\leq n$. Then there is some $\eta>0$ such that
$q(x)\geq\eta$ for all $x\in S_{X}$. By scaling, we may assume that $\eta=1 $.
Suppose that $y\in X$ satisfies $\left\Vert y\right\Vert _{X}<1$. Then
$y=\alpha x$ where $|\alpha|<1$ and $x\in S_{X}$. We compute%

\begin{align*}
q(y)  &  =q(\alpha x)=q_{1}(\alpha x)+q_{2}(\alpha x)+\cdots+q_{n}(\alpha x)\\
&  =\alpha^{2}q_{1}(x)+\alpha^{4}q_{2}(x)+\cdots+\alpha^{2n}q_{n}(x)\\
&  \geq\alpha^{2n}q(x)=\left\Vert y\right\Vert _{X}^{2n}q(x)\geq\left\Vert
y\right\Vert _{X}^{2n}\text{.}%
\end{align*}

\smallskip

Now, suppose that $y=\alpha x$ and $x\in S_{X}$ with no constraint on $\alpha
$, and that $q(y)<1$. Then $\alpha^{2i}q_{i}(x)<1$ \ for all $i$. Were
$\alpha\geq1$ we would have $1>q(y)\geq q(x)$ contradicting $q(x)\geq1$. Thus
$\alpha<1$ and we obtain $\left\Vert y\right\Vert _{X}^{2n}\leq q(y)$. For the
second statement, let $q_{i}(x)=M_{i}(x,x,\ldots,x)$ for $1\leq i\leq n$,
where $M_{i}$ is a bounded$\ $and $2i$-homogeneous multilinear functional for
each such $i$. Then if $y=\alpha x$ for $x\in S_{X}$ we have%
\[
q_{i}(y)=q_{i}(\alpha x)=\alpha^{2i}M_{i}(x,x,\ldots,x)\leq\alpha^{2i}A_{i}%
\]
where $\sup\{M_{i}(x,x,\ldots,x):x\in B_{X}\}=A_{i}$. So%
\[
q(y)=\sum_{i=1}^{n}q_{i}(y)\leq\sum_{i=1}^{n}A_{i}\left\Vert y\right\Vert
_{X}^{2i}\text{. }%
\]
Let $K=\sum_{i=1}^{n}A_{i}$. Observing that $\left\Vert y\right\Vert _{X}%
^{2i}\leq\max\{\left\Vert y\right\Vert _{X},\left\Vert y\right\Vert _{X}%
^{2n}\}$ for each $i$, the required inequality follows.
\end{proof}

\medskip

\section{Main Results}

Let us remark that if $F$ is uniformly approximated by Lipschitz functions,
then a necessary condition on $F$ is that it be uniformly continuous. Also, if
the norm on $X$ can be uniformly approximated by analytic functions, then it
is easy to show that $X$ admits a $C^{\infty}$-smooth bump function. Finally,
we observe that if $G$ is star-shaped as well as bounded, then $F$ is
necessarily bounded on $G$ as it is uniformly continuous.

\medskip

Without loss of generality, for the remainder of this paper we may and do
assume that $0\in G.$ Let $q$ be the polynomial constructed in Lemma
\ref{polynomial}. We extend $q$ to a complex analytic map $\widetilde{q}$ on
$\widetilde{X}$. Let $M\geq1$ be such that $q<M$ on $G-G.$ Such an $M$ exists
because $G$ is bounded. Note that being the extension of a polynomial to
$\widetilde{X},$ $\widetilde{q}$ is Lipschitz on bounded neighbourhoods of
$G-G$ in $\widetilde{X}.$ For the remainder of the paper, let us fix such a
bounded neighbourhood $\widetilde{G}-G$ where $\widetilde{G}\subset
\widetilde{X}$ is a bounded neighbourhood containing $G$, and choose
$\widetilde{M}$ so that $\left\vert \widetilde{q}\right\vert \leq\widetilde
{M}$ on $\widetilde{G}-G$. Also, fix $R>1$ so that $G\subset B_{R}\left(
0\right)  .$

Since $F$ is bounded and real-valued, by composing with an appropriate first
order polynomial, we may suppose without loss of generality that $1\geq
F\geq1/3$ on $G$. Fix $\varepsilon\in\left(  0,1/4\right)  $ for the rest of
the proof, and fix $\delta>0$ so that by uniform continuity, for $x,y\in G$ we
have%
\[
\left\Vert x-y\right\Vert _{X}<\delta\text{ \ implies \ }\left\vert
F(x)-F(y)\right\vert <\varepsilon\text{.}%
\]
Let $\gamma_{1}<$ $\gamma_{2}<\gamma_{3}$ be positive with $\gamma_{3}<1$ so
small that $q(y)<\gamma_{3}$ implies $\left\Vert y\right\Vert _{X}<\delta$ and
such that $3\gamma_{1}<\gamma_{2}/2<1$. Using separability, let $\{x_{j}%
\}_{j=1}^{\infty}$ be a dense subset of $G$, and define three open coverings
of $G$ using the sets%
\[
C_{j}^{i}=\left\{  x\in X:q(x-x_{j})<\gamma_{i}\right\}  ,
\]
for $j\in\mathbb{N}$ and $i=1,2,3$. That these form open covers follows from
Lemma 2.

\medskip

\subsection{\noindent The functionals $\widetilde{\varphi}_{n}$}

\medskip

\noindent The purpose of this section is to establish the following key lemma.
We use the notation established above.

\begin{lemma}
For $\eta>0,$ there exists a sequence of complex analytic maps $\widetilde
{\varphi}_{n}:\widetilde{X}\rightarrow\mathbb{C}$ with the following properties:

\medskip

\begin{enumerate}
\item[(i).] The collection $\left\{  \widetilde{\varphi}_{n}\mid_{X}\right\}
$ is Uniformly Lipschitz (UL) on $G,$ with Lipschitz constant independent of
$\eta.$

\medskip

\item[(ii).] For each $x_{0}\in G,$ there exists $n$ with $\widetilde{\varphi
}_{n}\left(  x_{0}\right)  >1/2.$

\medskip

\item[(iii).] For each $x_{0}\in G,$ there exist $\delta>0$ and $n_{0}>1$ such
that for $\left\Vert z\right\Vert _{\widetilde{X}}<\delta$ and $n>n_{0}$ we have%

\[
\left\vert \widetilde{\varphi}_{n}\left(  x_{0}+z\right)  \right\vert <\eta.
\]

\end{enumerate}
\end{lemma}

\medskip

\noindent{\small P}{\scriptsize ROOF}\textsc{.}\textbf{\ \ }Define subsets
\[
A_{n}=\left\{  y=\left\{  y_{j}\right\}  _{j=1}^{n}\in l_{\infty}^{n}%
:2\gamma_{1}<y_{j}<M+2\text{ for }1\leq j\leq n\right\}
\]
and
\[
A_{n}^{\prime}=\left\{  y=\left\{  y_{j}\right\}  _{j=1}^{n}\in l_{\infty}%
^{n}:3\gamma_{1}<y_{j}<M+1\text{ for }1\leq j\leq n\right\}  .
\]

\medskip

Let $b\in C^{\infty}\left(  \mathbb{R},\left[  0,1\right]  \right)  $ such
that $b\left(  t\right)  =1$ iff $t\notin\left(  2\gamma_{1},M+2\right)  ,$
and $b\left(  t\right)  =0$ iff $t\in\left[  3\gamma_{1},M+1\right]  .$ Now
define $b_{n}:l_{\infty}^{n}\rightarrow\left[  0,1\right]  $ by $b_{n}\left(
y_{1},...,y_{n}\right)  =1-\left\Vert \left(  b\left(  y_{1}\right)
,...,b\left(  y_{n}\right)  \right)  \right\Vert _{\infty}.$ Then
support$\left(  b_{n}\right)  =\overline{A}_{n},$ and $b_{n}=1$ on
$A_{n}^{\prime}.$ Moreover, $b_{n}$ is Lipschitz with constant $L_{b}$
independent of $n.$

\medskip

Define $\nu_{n}:l_{\infty}^{n}\rightarrow\mathbb{R}$ by%
\[
\nu_{n}\left(  x\right)  =\frac{1}{T_{n}}\int_{\mathbb{R}^{n}}b_{n}\left(
y\right)  e^{-\kappa_{n}\sum_{j=1}^{n}2^{-j}\left(  x_{j}-y_{j}\right)  ^{2}%
}dy,
\]
with
\begin{align*}
T_{n}  &  =\int_{\mathbb{R}^{n}}e^{-\kappa_{n}\sum_{j=1}^{n}2^{-j}y_{j}^{2}%
}dy\\
& \\
&  =\frac{1}{\kappa_{n}^{n/2}}\int_{\mathbb{R}^{n}}e^{-\sum_{j=1}^{n}%
2^{-j}y_{j}^{2}}dy\\
& \\
&  \equiv\frac{1}{\kappa_{n}^{n/2}}\ \widehat{T}_{n},
\end{align*}
and where the constants $\kappa_{n}^{n/2}\geq\left(  n!\right)  ^{2}\left(
\frac{\widehat{T}_{n}}{Vol\left(  A_{n}\right)  }\right)  $ shall be further
specified later.

Note that
\begin{align*}
\nu_{n}\left(  x\right)   &  =\frac{1}{T_{n}}\int_{\mathbb{R}^{n}}b_{n}\left(
y\right)  e^{-\kappa_{n}\sum_{j=1}^{n}2^{-j}\left(  x_{j}-y_{j}\right)  ^{2}%
}dy\\
& \\
&  =\frac{1}{T_{n}}\int_{\mathbb{R}^{n}}b_{n}\left(  x-y\right)
e^{-\kappa_{n}\sum_{j=1}^{n}2^{-j}y_{j}^{2}}dy,
\end{align*}
and so%

\begin{align*}
\left\vert \nu_{n}\left(  x\right)  -\nu_{n}\left(  x^{\prime}\right)
\right\vert  &  =\left\vert \frac{1}{T_{n}}\int_{\mathbb{R}^{n}}\left(
b_{n}\left(  x-y\right)  -b_{n}\left(  x^{\prime}-y\right)  \right)
e^{-\kappa_{n}\sum_{j=1}^{n}2^{-j}y_{j}^{2}}dy\right\vert \\
& \\
&  \leq\frac{1}{T_{n}}\int_{\mathbb{R}^{n}}\left\vert b_{n}\left(  x-y\right)
-b_{n}\left(  x^{\prime}-y\right)  \right\vert e^{-\kappa_{n}\sum_{j=1}%
^{n}2^{-j}y_{j}^{2}}dy\\
& \\
&  \leq L_{b}\ \left\Vert x-x^{\prime}\right\Vert _{\infty}\frac{1}{T_{n}}%
\int_{\mathbb{R}^{n}}e^{-\kappa_{n}\sum_{j=1}^{n}2^{-j}y_{j}^{2}}dy\\
& \\
&  =L_{b}\left\Vert x-x^{\prime}\right\Vert _{\infty}.
\end{align*}
Hence, $\nu_{n}$ is $L_{b}$-Lipschitz.

\medskip

\noindent Next, consider the map $\lambda_{n}:G\rightarrow l_{\infty}^{n}$
given by
\[
\lambda_{n}\left(  x\right)  =\left(  q\left(  x-x_{1}\right)  ,...,q\left(
x-x_{n}\right)  \right)  .
\]
\noindent Then for $n\geq1$ define (real) analytic maps $\varphi
_{n}:G\rightarrow\mathbb{R}$ by

\medskip%

\begin{align*}
\varphi_{n}\left(  x\right)   &  =\nu_{n}\left(  \lambda_{n}\left(  x\right)
\right)  =\nu_{n}\left(  \left\{  q\left(  x-x_{j}\right)  \right\}
_{j=1}^{n}\right) \\
& \\
&  =\frac{1}{T_{n}}\int_{\mathbb{R}^{n}}b_{n}(y)e^{-\kappa_{n}\sum_{j=1}%
^{n}2^{-j}\left(  q\left(  x-x_{j}\right)  -y_{j}\right)  ^{2}}dy\text{.}%
\end{align*}
Also, define $\varphi_{0}(x)=1$ for all $x$. Observe that
\begin{equation}
\int_{\mathbb{R}^{n}}e^{-\kappa_{n}\sum_{j=1}^{n}2^{-j}\left(  q\left(
x-x_{j}\right)  -y_{j}\right)  ^{2}}dy=\int_{\mathbb{R}^{n}}e^{-\kappa_{n}%
\sum_{j=1}^{n}2^{-j}y_{j}^{2}}dy=T_{n},
\end{equation}

\noindent and so in particular for any $n$ (using $b_{n}\in\left[  0,1\right]
$)%
\begin{align*}
\varphi_{n}\left(  0\right)   &  =\frac{1}{T_{n}}\int_{\mathbb{R}^{n}}%
b_{n}\left(  y\right)  e^{-\kappa_{n}\sum_{j=1}^{n}2^{-j}\left(  q\left(
-x_{j}\right)  -y_{j}\right)  ^{2}}dy\\
& \\
&  \leq\frac{1}{T_{n}}\int_{\mathbb{R}^{n}}e^{-\kappa_{n}\sum_{j=1}^{n}%
2^{-j}\left(  q\left(  -x_{j}\right)  -y_{j}\right)  ^{2}}dy\\
& \\
&  =\frac{1}{T_{n}}\int_{\mathbb{R}^{n}}e^{-\kappa_{n}\sum_{j=1}^{n}%
2^{-j}y_{j}^{2}}dy=1.
\end{align*}

\medskip

\noindent Now, as to the Lipschitz properties of the $\varphi_{n},$ recall
that since $G$ is bounded and $q$ is a polynomial, $q$ is Lipschitz on $G-G$
with constant, say, $L_{q}.$ Now%

\begin{align*}
\left\vert \varphi_{n}\left(  x\right)  -\varphi_{n}\left(  x^{\prime}\right)
\right\vert  &  =\left\vert \nu_{n}\left(  \lambda_{n}\left(  x\right)
\right)  -\nu_{n}\left(  \lambda_{n}\left(  x^{\prime}\right)  \right)
\right\vert \\
& \\
&  \leq L_{b}\left\Vert \lambda_{n}\left(  x\right)  -\lambda_{n}\left(
x^{\prime}\right)  \right\Vert _{\infty}\\
& \\
&  =L_{b}\left\Vert \left\{  q\left(  x-x_{j}\right)  -q\left(  x^{\prime
}-x_{j}\right)  \right\}  _{j=1}^{n}\right\Vert _{\infty}\\
& \\
&  \leq L_{b}L_{q}\left\Vert x-x^{\prime}\right\Vert _{X},
\end{align*}
and so the collection $\left\{  \varphi_{n}\right\}  $ is UL on $G$ with
constant $L_{b}L_{q}$. Using this fact, and $\varphi_{n}\left(  0\right)
\leq1,$ we put $W_{1}=L_{b}L_{q}R+1,$ and note that $\left\vert \varphi
_{n}(x)\right\vert \leq W_{1}$ for all $x\in G\subset B_{R}\left(  0\right)  $
and all $n\geq0$. In particular, $W_{1}$ is independent of $\kappa_{n}.$

\medskip

\noindent We extend the maps $\varphi_{n}$ to complex valued maps on
$\widetilde{X},$ calling them $\widetilde{\varphi}_{n}.$ Namely (where $x\in
X$)%

\[
\widetilde{\varphi}_{n}\left(  x+z\right)  =\frac{1}{T_{n}}\int_{\mathbb{R}%
^{n}}b_{n}\left(  y\right)  e^{-\kappa_{n}\sum_{j=1}^{n}2^{-j}\left(
\widetilde{q}\left(  x-x_{j}+z\right)  -y_{j}\right)  ^{2}}dy
\]

\noindent Note that the $\widetilde{\varphi}_{n}$ are complex analytic, and
that the above calculation establishes (i) as $\widetilde{\varphi}_{n}\mid
_{X}=\varphi_{n}.$

\medskip

\noindent We next show (iii). For $x\in X$ and $z\in\widetilde{X},$ we have
(see \cite{K} which references \cite{H})%
\[
\widetilde{q}\left(  x-x_{j}+z\right)  =q\left(  x-x_{j}\right)  +Z_{j},
\]
where $Z_{j}\in\mathbb{C}$ with $\left\vert Z_{j}\right\vert \leq
C(1+\left\Vert x-x_{j}\right\Vert _{X})^{m}\left\Vert z\right\Vert
_{\widetilde{X}},$ $C$ being a constant and $m$ the degree of $q.$ As $G$ is
bounded, we can bound $C(1+\left\Vert x-x_{j}\right\Vert _{X})^{m}$ above by
some constant $M_{1}.$

\medskip

Now%
\begin{align*}
\left(  \widetilde{q}\left(  x-x_{j}+z\right)  -y_{j}\right)  ^{2}  &
=\left(  q\left(  x-x_{j}\right)  -y_{j}+Z_{j}\right)  ^{2}\\
&  =\left(  q\left(  x-x_{j}\right)  -y_{j}\right)  ^{2}+2\left(  q\left(
x-x_{j}\right)  -y_{j}\right)  Z_{j}+Z_{j}^{2}.
\end{align*}
Hence%
\begin{align*}
&  \operatorname{Re}\left(  \widetilde{q}\left(  x-x_{j}+z\right)
-y_{j}\right)  ^{2}\\
& \\
&  =\left(  q\left(  x-x_{j}\right)  -y_{j}\right)  ^{2}+2\left(  q\left(
x-x_{j}\right)  -y_{j}\right)  \operatorname{Re}Z_{j}+\operatorname{Re}%
Z_{j}^{2}\\
& \\
&  \geq\left(  q\left(  x-x_{j}\right)  -y_{j}\right)  ^{2}-2\left\vert
q\left(  x-x_{j}\right)  -y_{j}\right\vert \left\vert Z_{j}\right\vert
-\left\vert Z_{j}^{2}\right\vert \\
& \\
&  =\left(  \left\vert q\left(  x-x_{j}\right)  -y_{j}\right\vert -\left\vert
Z_{j}\right\vert \right)  ^{2}-2\left\vert Z_{j}\right\vert ^{2}\\
& \\
&  \geq\left(  \left\vert q\left(  x-x_{j}\right)  -y_{j}\right\vert
-\left\vert Z_{j}\right\vert \right)  ^{2}-2M_{1}^{2}\left\Vert z\right\Vert
_{\widetilde{X}}^{2}.
\end{align*}

Next, for each $x_{0}\in G$ there exists $j_{0}$ so that $x_{0}\in C_{j_{0}%
}^{1}$ and thus $q\left(  x_{0}-x_{j_{0}}\right)  <\gamma_{1},$ and also for
$y\in A_{j_{0}}\subset\ $support$\left(  b_{j_{0}}\right)  $ we have
$y_{j_{0}}>2\gamma_{1}.$ Hence, $y_{j_{0}}-q\left(  x_{0}-x_{j_{0}}\right)
>2\gamma_{1}-\gamma_{1}>\gamma_{1}.$ Thus, for $\left\Vert z\right\Vert
_{\widetilde{X}}<\frac{\gamma_{1}}{2M_{1}}$ we have $\left\vert q\left(
x_{0}-x_{j_{0}}\right)  -y_{j_{0}}\right\vert -\left\vert Z_{j_{0}}\right\vert
\geq\left\vert q\left(  x_{0}-x_{j_{0}}\right)  -y_{j_{0}}\right\vert
-M_{1}\left\Vert z\right\Vert _{\widetilde{X}}>\gamma_{1}/2.$ It follows that
for $\left\Vert z\right\Vert _{\widetilde{X}}<\frac{\gamma_{1}}{2M_{1}}$ and
$n>j_{0}$, that%

\begin{align*}
&  \sum_{j=1}^{n}2^{-j}\operatorname{Re}\left(  \widetilde{q}\left(
x_{0}-x_{j}+z\right)  -y_{j}\right)  ^{2}\\
& \\
&  \geq\sum_{j=1}^{n}2^{-j}\left(  \left(  \left\vert q\left(  x_{0}%
-x_{j}\right)  -y_{j}\right\vert -\left\vert Z_{j}\right\vert \right)
^{2}-2M_{1}^{2}\left\Vert z\right\Vert _{\widetilde{X}}^{2}\right) \\
& \\
&  \geq2^{-j_{0}}\left(  \left\vert q\left(  x_{0}-x_{j_{0}}\right)
-y_{j_{0}}\right\vert -\left\vert Z_{j_{0}}\right\vert \right)  ^{2}%
-2M_{1}^{2}\sum_{j=1}^{n}2^{-j}\left\Vert z\right\Vert _{\widetilde{X}}^{2}\\
& \\
&  >2^{-j_{0}-1}\gamma_{1}-2M_{1}^{2}\left\Vert z\right\Vert _{\widetilde{X}%
}^{2}.
\end{align*}
Hence, there exists $\delta>0$ and $a>0$ such that $\left\Vert z\right\Vert
_{X}<\delta$ and $n>j_{0}$ implies $\sum_{j=1}^{n}2^{-j}\operatorname{Re}%
\left(  \widetilde{q}\left(  x_{0}-x_{j}+z\right)  -y_{j}\right)  ^{2}>a.$

\medskip

\noindent Putting this together shows that for each $x_{0}\in X$ there exists
$\widehat{n}_{0}>j_{0},$ $\delta>0$ and $a>0$ so that for all $n>\widehat
{n}_{0}$ and $\left\Vert z\right\Vert _{\widetilde{X}}<\delta,$ we have%

\begin{align*}
\left\vert \widetilde{\varphi}_{n}\left(  x_{0}+z\right)  \right\vert  &
=\left\vert \frac{1}{T_{n}}\int_{\mathbb{R}^{n}}b_{n}\left(  y\right)
e^{-\kappa_{n}\sum_{j=1}^{n}2^{-j}\left(  \widetilde{q}\left(  x_{0}%
-x_{j}+z\right)  -y_{j}\right)  ^{2}}dy\right\vert \\
& \\
&  \leq\frac{1}{T_{n}}\int_{\mathbb{R}^{n}}b_{n}\left(  y\right)
e^{-\kappa_{n}\operatorname{Re}\sum_{j=1}^{n}2^{-j}\left(  \widetilde
{q}\left(  x_{0}-x_{j}+z\right)  -y_{j}\right)  ^{2}}dy\\
& \\
&  \leq\frac{1}{T_{n}}\int_{A_{n}}e^{-\kappa_{n}a}dy=\frac{Vol\left(
A_{n}\right)  }{T_{n}}e^{-\kappa_{n}a}.
\end{align*}
\noindent Now, by choice of $\kappa_{n}$ (see above) and the definition of
$\widehat{T}_{n}$ (which is independent of $\kappa_{n}$)%
\[
\frac{Vol\left(  A_{n}\right)  }{T_{n}}\ e^{-\kappa_{n}a}<\frac{Vol\left(
A_{n}\right)  }{T_{n}}\ \frac{n!}{\kappa_{n}^{n}a^{n}}=\frac{Vol\left(
A_{n}\right)  }{\widehat{T}_{n}}\ \frac{n!}{\kappa_{n}^{n/2}a^{n}}<\frac
{1}{n!a^{n}},
\]

\noindent and so by choosing $n_{0}>\widehat{n}_{0}$ sufficiently large, we
can guarantee that for $n>n_{0}$ and $\left\Vert z\right\Vert _{\widetilde{X}%
}<\delta$ we have $\left\vert \widetilde{\varphi}_{n}\left(  x_{0}+z\right)
\right\vert <\eta$, and so (iii) is proven.

\medskip

\noindent Finally we show (ii). Now, since $A_{n}^{\prime}\subset A_{n},$
$b_{n}=1$ on $A_{n}^{\prime},$ and the integrand is positive$,$ we have
\begin{align*}
&  \frac{1}{T_{n}}\int_{A_{n}^{\prime}}e^{-\kappa_{n}\sum_{j=1}^{n}%
2^{-j}\left(  q\left(  x-x_{j}\right)  -y_{j}\right)  ^{2}}dy\\
& \\
&  =\frac{1}{T_{n}}\int_{A_{n}^{\prime}}b_{n}\left(  y\right)  e^{-\kappa
_{n}\sum_{j=1}^{n}2^{-j}\left(  q\left(  x-x_{j}\right)  -y_{j}\right)  ^{2}%
}dy\\
& \\
&  <\frac{1}{T_{n}}\int_{A_{n}}b_{n}\left(  y\right)  e^{-\kappa_{n}\sum
_{j=1}^{n}2^{-j}\left(  q\left(  x-x_{j}\right)  -y_{j}\right)  ^{2}}dy.
\end{align*}
Finally, using $b_{n}\leq1,$ we observe that $\int_{\mathbb{R}^{n}}%
e^{-\kappa_{n}\sum_{j=1}^{n}2^{-j}\left(  q\left(  x-x_{j}\right)
-y_{j}\right)  ^{2}}dy-\int_{A_{n}}b_{n}\left(  y\right)  e^{-\kappa_{n}%
\sum_{j=1}^{n}2^{-j}\left(  q\left(  x-x_{j}\right)  -y_{j}\right)  ^{2}%
}dy\geq0.$ Thus, putting this together and using $\left(  3.1\right)  $ gives

\medskip%

\begin{align}
&  1-\varphi_{n}\left(  x\right) \nonumber\\
& \nonumber\\
&  =\frac{1}{T_{n}}\int_{\mathbb{R}^{n}}e^{-\kappa_{n}\sum_{j=1}^{n}%
2^{-j}\left(  q\left(  x-x_{j}\right)  -y_{j}\right)  ^{2}}dy\nonumber\\
& \nonumber\\
&  -\frac{1}{T_{n}}\int_{A_{n}}b_{n}\left(  y\right)  e^{-\kappa_{n}\sum
_{j=1}^{n}2^{-j}\left(  q\left(  x-x_{j}\right)  -y_{j}\right)  ^{2}%
}dy\nonumber\\
& \\
&  \leq\frac{1}{T_{n}}\left(  \int_{\mathbb{R}^{n}}e^{-\kappa_{n}\sum
_{j=1}^{n}2^{-j}\left(  q\left(  x-x_{j}\right)  -y_{j}\right)  ^{2}}dy\right.
\nonumber\\
& \nonumber\\
&  \left.  -\int_{A_{n}^{\prime}}e^{-\kappa_{n}\sum_{j=1}^{n}2^{-j}\left(
q\left(  x-x_{j}\right)  -y_{j}\right)  ^{2}}dy\right) \nonumber\\
& \nonumber\\
&  =\frac{1}{T_{n}}\int_{\mathbb{R}^{n}\backslash A_{n}^{\prime}}%
e^{-\kappa_{n}\sum_{j=1}^{n}2^{-j}\left(  q\left(  x-x_{j}\right)
-y_{j}\right)  ^{2}}dy\nonumber
\end{align}

Next, for each fixed $x_{0}\in G,$ there exists $j_{0}$ with $x_{0}\in
C_{j_{0}}^{2}$ but with $x_{0}\notin G\backslash C_{i}^{2}$ for $1\leq
i<j_{0}$. Hence, $q\left(  x_{0}-x_{j}\right)  \geq\gamma_{2}$ for all $1\leq
j<j_{0}.$ Now suppose that $\left\Vert \left\{  q\left(  x_{0}-x_{j}\right)
-y_{j}\right\}  _{j<j_{0}}\right\Vert _{\infty}<\gamma_{2}/2.$ Then
$y_{j}>q\left(  x_{0}-x_{j}\right)  -\gamma_{2}/2\geq\gamma_{2}/2>3\gamma_{1}$
for $1\leq j<j_{0}.$ Also, $y_{j}<q\left(  x_{0}-x_{j}\right)  +\gamma
_{2}/2\leq M+1.$ Hence, putting $p_{0}=\left\{  q\left(  x_{0}-x_{j}\right)
\right\}  _{j<j_{0}}\in l_{\infty}^{j_{0}-1},$ we have $B_{\gamma_{2}%
/2}\left(  p_{0}\right)  \subset A_{j_{0}-1}^{\prime}.$ Thus, writing
$q\left(  x_{0}-x_{j}\right)  =\left(  p_{0}\right)  _{j},$ and for ease of
notation $n=j_{0}-1,$ we have, using $(3.2)$%

\begin{align*}
\left\vert 1-\varphi_{n}\left(  x_{0}\right)  \right\vert  &  \leq\frac
{1}{T_{n}}\int_{\mathbb{R}^{n}\backslash B_{\gamma_{2}/2}\left(  p_{0}\right)
}e^{-\kappa_{n}\sum_{j=1}^{n}2^{-j}\left(  \left(  p_{0}\right)  _{j}%
-y_{j}\right)  ^{2}}dy\\
&  =\frac{1}{T_{n}}\int_{\mathbb{R}^{n}\backslash B_{\gamma_{2}/2}\left(
0\right)  }e^{-\kappa_{n}\sum_{j=1}^{n}2^{-j}y_{j}^{2}}dy.
\end{align*}
By choosing $\kappa_{n}$ larger if necessary, we can guarantee that
\[
\frac{1}{T_{n}}\int_{\mathbb{R}^{n}\backslash B_{\gamma_{2}/2}\left(
0\right)  }e^{-\kappa_{n}\sum_{j=1}^{n}2^{-j}y_{j}^{2}}dy<1/2,
\]
and so $\varphi_{n}\left(  x_{0}\right)  >1/2.$ If however $j_{0}=1$ then
$\varphi_{j_{0}-1}(x_{0})=\varphi_{0}(x_{0})=1>1/2$. $\square$

\medskip

\subsection{The functionals $\psi_{j}$}

Let $\zeta^{1}:\mathbb{[}0,M]\rightarrow\mathbb{R}$ satisfy $\zeta^{1}\left(
t\right)  \geq1/8$ for $t\in\left[  0,M\right]  $, be real analytic and
Lipschitz, and be such that $\zeta^{1}\left(  t\right)  <1/4$ if $t\leq
\gamma_{2},$ and $\zeta^{1}\left(  t\right)  \geq1$ if $t\geq\gamma_{3}.$ We
can choose $\zeta^{1}$ to be a polynomial, and hence it can be extended as a
complex analytic map on $\mathbb{C},$ Lipschitz on $B_{\widetilde{M}}\left(
0\right)  \subset\mathbb{C},$ and if so we keep the same notation. Using this,
define on $G$, for each $j\geq1$%
\[
f_{j}(x)=\zeta^{1}\left(  q\left(  x-x_{j}\right)  \right)  .
\]
Each $f_{j}$ is real analytic on $G,$ and the collection $\left\{
f_{j}\right\}  $ is UL on $G$ since as pointed out above, $q$ is bounded (and
Lipschitz) on $G-G$, and $\zeta^{1}$ is Lipschitz. Let us write
\[
\widetilde{f}_{j}\left(  z\right)  =\zeta^{1}\left(  \widetilde{q}\left(
z-x_{j}\right)  \right)  \text{,}%
\]
noting that similarly $\left\{  \widetilde{f}_{j}\right\}  $ is UL on
$\widetilde{G}.$

\medskip

Next, for $j\geq1$ define $\zeta^{2}:\mathbb{[}0,W_{1}]\rightarrow$
$\mathbb{R}$ to be strictly positive with $\zeta^{2}\left(  t\right)  \geq1/8$
for $t\in\left[  0,W_{1}\right]  $, real analytic, Lipshitz, and such that
$\zeta^{2}\left(  t\right)  \geq2\ $if $t\leq1/4,\ \zeta^{2}\left(  t\right)
<1/4$ if $t\geq1/2.\ $We can choose $\zeta^{2}$ to be a polynomial, and hence
it can be extended as a complex analytic map on $\mathbb{C},$ which is
Lipschitz on $B_{W_{1}+1}\left(  0\right)  \subset\mathbb{C}$ with constant
$L_{2}\geq1,$ and if so we keep the same notation. For each $j\geq1$ define a
real analytic mapping on $G$ by%

\[
g_{j}(x)=\zeta^{2}\left(  \varphi_{j-1}\left(  x\right)  \right)  .
\]

\noindent And for each $j\geq1$ define a complex analytic mapping on
$\widetilde{X}$ by%

\[
\widetilde{g}_{j}(x)=\zeta^{2}\left(  \widetilde{\varphi}_{j-1}\left(
x\right)  \right)  .
\]

\noindent The collection $\left\{  g_{j}\right\}  $ is UL on $G$ as the same
is true of the $\varphi_{j}.$

\medskip

\noindent It will be convenient to define the constant $T$ by%
\[
T=\sup\left\{  |z|:z\in\zeta^{1}(B_{\widetilde{M}}(0))\right\}  +\sup\left\{
|z|:z\in\zeta^{2}(B_{W_{1}+1}(0))\right\}  \text{.}%
\]

\medskip

\noindent Let $h:[0,T]\rightarrow\left[  0,1\right]  $ be strictly greater
than $0$, real analytic, Lipschitz, and such that $h\left(  t\right)
<\varepsilon/4<1/10$ if $t\geq3/4,$ and $h\left(  t\right)  \geq4/5$ if
$t\leq1/2.$ We can choose $h$ to be a polynomial, and hence it can be extended
as a complex analytic map on $\mathbb{C},$ Lipschitz on $B_{T}\left(
0\right)  \subset\mathbb{C}$ with constant $L_{h}\geq1$, and if so we keep the
same notation.

\noindent Now define maps%

\[
\psi_{j}(x)=f_{j}\left(  x\right)  +g_{j}\left(  x\right)  \text{,}%
\]

\noindent and the maps%

\[
\widetilde{\psi}_{j}(x)=\widetilde{f}_{j}\left(  x\right)  +\widetilde{g}%
_{j}\left(  x\right)  .
\]

\medskip

Note that on $G$, each $\psi_{j}$ is analytic and the collection $\left\{
\psi_{j}\right\}  _{j}$ is UL on $G$. We pause to summarize the important
properties of the $\psi_{j}$ functions in a lemma.

\begin{lemma}
\label{Mainlemma}For the functions $\psi_{j}$ defined above, we have
\end{lemma}

\begin{enumerate}
\item[(i)] $\psi_{j}(x)\geq1$\textit{\ for all }$x\in G\backslash C_{j}^{3}
$\textit{.}

\medskip

\item[(ii)] \textit{For each }$x\in G$\textit{, there is a }$j_{0}%
$\textit{\ such that }$\psi_{j_{0}}(x)<1/2$\textit{.}

\medskip

\item[(iii)] \textit{For }$x\in G$\textit{, }$z\in\widetilde{X},$
\textit{there is a }$j_{x}$\textit{\ and a }$\delta=\delta_{x}>0$%
\textit{\ such that }%
\[
\left\vert \widetilde{\psi}_{j}\left(  x+z\right)  -\widetilde{\psi}%
_{j}\left(  x\right)  \right\vert <1/(10L_{h})
\]
\textit{\ for }$\left\Vert z\right\Vert _{\widetilde{X}}<\delta$\textit{\ and
}$j>j_{x}$\textit{. Moreover, for }$j>j_{x}$ we have $\psi_{j}\left(
x\right)  >1$.

\medskip
\end{enumerate}

\begin{proof}
(i) For any $x\in G$%
\[
\psi_{j}(x)=f_{j}\left(  x\right)  +g_{j}\left(  x\right)  \geq f_{j}\left(
x\right)  =\zeta_{j}^{1}(q(x-x_{j}))\text{.}%
\]
\noindent But if $x\in G\backslash C_{j}^{3}$, $q(x-x_{j})\geq\gamma_{3}$,
implying $\zeta^{1}(q(x-x_{j}))\geq1$.

\medskip

(ii) Fix $x\in G$. There is a $j_{0}$ with $x\in C_{j_{0}}^{2}$ but with $x\in
G\backslash C_{i}^{2}$ for $1\leq i<j_{0}$. From the construction of the
$f_{j}$ and $g_{j}$ and using Lemma 3 (ii),\ $f_{j_{0}}(x)<1/4$ and $g_{j_{0}%
}(x)<1/4$.

\medskip

(iii) Fix $x\in G,$ and choose $\delta_{1}$ satisfying $1\geq\delta_{1}>0$ so
that $\left\Vert z\right\Vert _{\widetilde{X}}<\delta_{1}$ implies
$x+z\in\widetilde{G}.$ Since $\{x_{j}\}$ is dense$,$ there is a $j_{x}%
^{^{\prime}}>1$ such that $x\in C_{j_{x}^{^{\prime}}}^{1}$. Now from Lemma 3
(iii) with $\eta=1/40L_{2}L_{h},$ we have that for $x\in C_{j_{x}^{^{\prime}}%
}^{1}$ there exists $\delta$ satisfying $0<\delta<\delta_{1}$ so that
$\left\Vert z\right\Vert _{\widetilde{X}}<\delta$ and $j>j_{x}^{^{\prime}}$
imply $\left\vert \widetilde{\varphi}_{j}\left(  x+z\right)  \right\vert
<\eta.$ Thus we have, for all $\left\Vert z\right\Vert _{\widetilde{X}}%
<\delta$ and all $j>j_{x}^{^{\prime}}$
\begin{align*}
|\widetilde{\varphi}_{j}\left(  x+z\right)  -\widetilde{\varphi}_{j}\left(
x\right)  |  &  \leq|\widetilde{\varphi}_{j}\left(  x+z\right)  |+|\widetilde
{\varphi}_{j}\left(  x\right)  |\\
& \\
&  <2\eta\leq1/\left(  20L_{2}L_{h}\right)  <1,
\end{align*}
and hence if $j>j_{x}\equiv j_{x}^{\prime}+1$
\begin{align*}
|\widetilde{g}_{j}\left(  x+z\right)  -\widetilde{g}_{j}\left(  x\right)  |
&  =|\zeta^{2}(\widetilde{\varphi}_{j-1}\left(  x+z\right)  )-\zeta
^{2}(\widetilde{\varphi}_{j-1}\left(  x\right)  )|\\
& \\
&  \leq L_{2}|\widetilde{\varphi}_{j-1}\left(  x+z\right)  -\widetilde
{\varphi}_{j-1}\left(  x\right)  |\\
& \\
&  \leq L_{2}\cdot1/\left(  20L_{2}L_{h}\right)  \leq1/\left(  20L_{h}\right)
.
\end{align*}
Now $f_{j}\left(  x\right)  \geq0$ and $\widetilde{f}_{j}$ is Lipschitz on $G$
with constant independent of $j,$ and so choosing $\delta>0$ again smaller if
necessary, we have $\left\vert \widetilde{f}_{j}(x+z)-\widetilde{f}_{j}\left(
x\right)  \right\vert <1/(20L_{h})$ for all $j$ and $\left\Vert z\right\Vert
_{\widetilde{X}}<\delta.$ We therefore have, for all $\left\Vert z\right\Vert
_{\widetilde{X}}<\delta$ and all $j>j_{x}$%
\begin{align}
&  \left\vert \widetilde{\psi}_{j}\left(  x+z\right)  -\widetilde{\psi}%
_{j}\left(  x\right)  \right\vert \nonumber\\
& \nonumber\\
&  =\left\vert \widetilde{f}_{j}\left(  x+z\right)  -\widetilde{f}_{j}\left(
x\right)  +\widetilde{g}_{j}\left(  x+z\right)  -\widetilde{g}_{j}\left(
x\right)  \right\vert \\
& \nonumber\\
&  \leq\left\vert \widetilde{f}_{j}\left(  x+z\right)  -\widetilde{f}%
_{j}\left(  x\right)  \right\vert +\left\vert \widetilde{g}_{j}\left(
x+z\right)  -\widetilde{g}_{j}\left(  x\right)  \right\vert \nonumber\\
& \nonumber\\
&  \leq1/(20L_{h})+1/(20L_{h})=1/\left(  10L_{h}\right)  \text{.}\nonumber
\end{align}
Finally, for $j>j_{x}$, $g_{j}(x)\geq2$ since $\varphi_{j-1}\left(  x\right)
=\left\vert \varphi_{j-1}\left(  x\right)  \right\vert <\eta\leq1/4$. And
since $f_{j}(x)\geq0$, $\psi_{j}\left(  x\right)  >1$.
\end{proof}

\subsection{Main Theorem}

We finally present the proof of Theorem 1, using the previous constructions,
results, and notation.

\medskip

\begin{proof}
[Proof of Theorem 1]We continue to use the notation introduced in the previous
lemmas. Define the real analytic map
\[
u_{j}\left(  x\right)  =h\left(  \psi_{j}\left(  x\right)  \right)  ,
\]
noting that on $G$, the collection $\left\{  u_{j}\right\}  $ is UL$.$ As
usual we write $\widetilde{u}_{j}\left(  z\right)  =h\left(  \widetilde{\psi
}_{j}\left(  z\right)  \right)  .$ Applying Lemma \ref{Mainlemma} (iii)$,$ we
obtain the following. For each $x\in G$ there exists $\delta=\delta\left(
x\right)  >0$ and $j_{x}>1$ so that for $z\in\widetilde{X}$ with $\left\Vert
z\right\Vert _{\widetilde{X}}<\delta$ we have for all $j\geq j_{x}$ that
$\left\vert \widetilde{u}_{j}\left(  x+z\right)  \right\vert <1/5$. Indeed,
using $\left(  3.3\right)  $%
\begin{align*}
\left\vert \widetilde{u}_{j}\left(  x+z\right)  -\widetilde{u}_{j}\left(
x\right)  \right\vert  &  =\left\vert h(\widetilde{\psi}_{j}\left(
x+z\right)  )-h(\widetilde{\psi}_{j}\left(  x\right)  )\right\vert \\
& \\
&  \leq L_{h}\left\vert \widetilde{\psi}_{j}\left(  x+z\right)  )-\widetilde
{\psi}_{j}\left(  x\right)  \right\vert \\
& \\
&  \leq L_{h}\cdot1/\left(  10L_{h}\right)  =1/10
\end{align*}

\medskip

But $\psi_{j}\left(  x\right)  =$ $\widetilde{\psi}_{j}\left(  x\right)  >1$
implying that $\left\vert \widetilde{u}_{j}\left(  x\right)  \right\vert
=h(\psi_{j}\left(  x\right)  )<1/10$. We conclude that%

\[
\left\vert \widetilde{u}_{j}\left(  x+z\right)  \right\vert \leq\left\vert
\widetilde{u}_{j}\left(  x\right)  \right\vert +1/10<1/10+1/10=1/5\text{.}%
\]

\medskip

A similar calculation with Lemma \ref{Mainlemma} (i) gives that $x\in
G\backslash C_{j}^{3}$ implies $u_{j}\left(  x\right)  <\varepsilon/4.$ Also,
by Lemma \ref{Mainlemma} (ii) we note that for each $x$ there is a $j_{0}$
with $\psi_{j_{0}}\left(  x\right)  <1/2,$ and hence $u_{j_{0}}\left(
x\right)  \geq4/5.$ Using the notation from section 2.1, the above shows that
for each $x\in G,$ $\left\{  u_{j}\left(  x\right)  \right\}  \in U,$ and
moreover, for any $x,y\in G$ and $a,b\in\mathbb{R}$ with $\left\vert
a\right\vert \leq15/8$ and $\left\vert b\right\vert \leq15/8$ we have%
\begin{align}
\{au_{j}(x)\}  &  \in U\\
& \nonumber\\
\{au_{j}(x)\}+\{bu_{j}(x)\}  &  \in U\nonumber
\end{align}
It follows from this and the subadditivity of $\lambda$ that for such $x,y$
and $a,b$ we have the 1-Lipschitz property
\begin{equation}
\left\vert \lambda\left(  \left\{  au_{j}\left(  x\right)  \right\}  \right)
-\lambda\left(  \left\{  bu_{j}\left(  y\right)  \right\}  \right)
\right\vert \leq\lambda\left(  \left\{  au_{j}\left(  x\right)  \right\}
-\left(  \left\{  bu_{j}\left(  y\right)  \right\}  \right)  \right)  \text{.}%
\end{equation}
Now put
\[
A\left(  x\right)  =\sum_{j=1}^{\infty}u_{j}\left(  x\right)  ^{2j}.
\]
We next show that $A$ is analytic on $G.$ From the calculation above, we have
that $\widetilde{A}\left(  x+z\right)  =\sum_{j=1}^{\infty}\widetilde{u}%
_{j}\left(  x+z\right)  ^{2j}$ converges locally uniformly at $x$. Since the
$\widetilde{u}_{j}$ are complex analytic, we conclude that $\widetilde{A}$ is
complex analytic in a neighbourhood of $x$ in $\widetilde{X}.$ Hence, since
$x\in G$ was arbitrary, in some neighbourhood $\widetilde{U}$ with
$G\subset\widetilde{U}\subset\widetilde{X}$ we have that $\widetilde{A}$ is
complex analytic and so $A\left(  x\right)  $ is real analytic on $G.$
Similarly from the calculations above, it follows that for each $x\in G$ the
map $H\left(  z,\mu\right)  =\sum_{j=1}^{\infty}\left(  \mu^{-1}\widetilde
{u}_{j}\left(  z\right)  \right)  ^{2j},$ is complex analytic on some
neighbourhood $\widetilde{N}_{x}=\widetilde{U}_{x}\times V,$ where
$\widetilde{U}_{x}\subset\widetilde{X}$ contains $x$ and $V=\left\{
w\in\mathbb{C}:\left\vert w\right\vert >7/30\right\}  $ (the lower bound
$7/30$ will be necessary later when we consider the sequence $\left\{
F\left(  x_{j}\right)  u_{j}\left(  x\right)  \right\}  $).

\medskip

Next, for fixed $x\in G,$ consider the map $Q:(8/15,\infty)\rightarrow U$
defined by $Q(\mu)=\sum_{j=1}^{\infty}\left(  \mu^{-1}u_{j}\left(  x\right)
\right)  ^{2j}.$ This map is continuous (see the argument above establishing
the analyticity of $H$). Also, for $\mu>\sqrt{2},$ $u_{j}\left(  x\right)
/\mu\leq1/\mu<1/\sqrt{2},$ and so $\sum_{j}\left(  u_{j}\left(  x\right)
/\mu\right)  ^{2j}\leq\sum_{j}\left(  \frac{1}{2}\right)  ^{j}<1.$ On the
other hand, as noted, there exists $j_{0}$ with $u_{j_{0}}\left(  x\right)
\geq4/5,$ and so if $\mu<5/4$ then $\sum_{j}\left(  u_{j}\left(  x\right)
/\mu\right)  ^{2j}\geq\left(  u_{j_{0}}\left(  x\right)  /\mu\right)
^{2j_{0}}>1.$ Hence, by the Intermediate Value Theorem, for each $x\in G$ the
equation $0=F\left(  x,\mu\right)  \equiv H\left(  x,\mu\right)  -1$ has a (in
this case unique) solution $\mu=\mu(x)$.

\medskip

Next note that $h$ is strictly bounded above $0,$ and so for each $x\in G$ it
follows from continuity that $\frac{\partial H\left(  z,\mu\right)  }%
{\partial\mu}\neq0$ in some neighbourhood $\widetilde{N}_{x}^{\prime
}=\widetilde{U}_{x}^{\prime}\times V^{\prime},$ where $\widetilde{U}%
_{x}^{\prime}\subset\widetilde{U}_{x}$ contains $x$ and $V^{\prime}\subset V$
contains $\mu(x).$ From the above considerations, and an argument analogous to
the one in section 2.1, it now follows by the complex Implicit Function
Theorem that on some neighbourhood $\widetilde{W}$ of $\left(  x,\mu
(x)\right)  $ in $\widetilde{X}\times\mathbb{C}$ there is a unique solution
$\mu\left(  z\right)  $ to the equation $F\left(  z,\mu\right)  =0$ which is
complex analytic on $\widetilde{W}$. In particular, its restriction to
$W=\widetilde{W}\cap\left(  G\times\mathbb{R}\right)  $ is real analytic.
Noting that $\left\{  u_{j}\left(  x\right)  \right\}  \in U,$ an examination
of the construction in section 2.1 of the extension of the Preiss norm to an
analytic function $\lambda$ on $U\subset l_{\infty},$ shows from the
uniqueness conclusion of the Implicit Function Theorem that $\mu\left(
x\right)  =\lambda\left(  \left\{  u_{j}\left(  x\right)  \right\}
_{j}\right)  $.

\medskip

Thus, the map $x\rightarrow\lambda\left(  \left\{  u_{j}\left(  x\right)
\right\}  \right)  $ is analytic on $G$. As well, from the previous estimates
on $\lambda$ it follows that
\begin{equation}
\frac{1}{2}\lambda\left(  \left\{  au_{j}\left(  x\right)  \right\}  \right)
\leq\left\Vert \left\{  au_{j}\left(  x\right)  \right\}  \right\Vert
_{\infty}\leq\lambda\left(  \left\{  au_{j}\left(  x\right)  \right\}
\right)  \text{,}%
\end{equation}
and these inequalities hold for any $a\in\mathbb{R}$ with $\left\{
au_{i}\left(  x\right)  \right\}  \in U.$ Finally define $K:G\rightarrow
\mathbb{R}$ by
\[
K(x)=\frac{\lambda\left(  \left\{  F\left(  x_{j}\right)  u_{j}\left(
x\right)  \right\}  \right)  }{\lambda\left(  \left\{  u_{j}\left(  x\right)
\right\}  \right)  }.
\]
for all $x\in G$. That the numerator of $K$ is (real) analytic follows the
same argument used to establish that the map $x\rightarrow\lambda\left(
\left\{  u_{j}\left(  x\right)  \right\}  \right)  $ is analytic, where now
one also uses the bounds $1\geq F\left(  x_{j}\right)  \geq1/3.$ Recall again
that for each $x$ there is a $j_{0}$ such that $u_{j_{0}}(x)\geq4/5.$ So,
using $(3.6)$ we obtain
\[
\lambda\left(  \{u_{j}\left(  x\right)  \}\right)  \geq\left\Vert
\{u_{j}\left(  x\right)  \}\right\Vert _{\infty}\geq u_{j_{0}}\left(
x\right)  \geq\frac{4}{5}\text{.}%
\]
Now the functions $F(x_{j})u_{j}\left(  x\right)  $ are Lipschitz with
constant independent of $j,$ since the collection $\left\{  u_{j}\right\}  $
is UL and $F$ is bounded. Also, by $\left(  3.5\right)  ,$ using $0<F\leq1,$
$\lambda$ is 1-Lipschitz on points of the form $\left\{  F\left(
x_{j}\right)  u_{j}\left(  x\right)  \right\}  _{j},$ and so the composition
$\lambda\left(  \left\{  F\left(  x_{j}\right)  u_{j}\left(  x\right)
\right\}  \right)  $ is Lipschitz$.$ These observations, combined with
$\lambda\left(  \{u_{j}\left(  x\right)  \}\right)  \geq4/5$ show that $K$ is
analytic and Lipschitz. Then using $\left(  3.4\right)  $ and $\left(
3.5\right)  $ we obtain the estimate%
\begin{align*}
\left\vert K(x)-F(x)\right\vert  &  =\left\vert \frac{\lambda\left(  \left\{
F\left(  x_{j}\right)  u_{j}\left(  x\right)  \right\}  \right)  }%
{\lambda\left(  \left\{  u_{j}\left(  x\right)  \right\}  \right)  }-F\left(
x\right)  \right\vert \\
& \\
&  =\left\vert \frac{\lambda\left(  \left\{  F\left(  x_{j}\right)
u_{j}\left(  x\right)  \right\}  \right)  }{\lambda\left(  \left\{
u_{j}\left(  x\right)  \right\}  \right)  }-\frac{F\left(  x\right)
\lambda\left(  \left\{  u_{j}\left(  x\right)  \right\}  \right)  }%
{\lambda\left(  \left\{  u_{j}\left(  x\right)  \right\}  \right)
}\right\vert \\
& \\
&  =\frac{1}{\lambda\left(  \left\{  u_{j}\left(  x\right)  \right\}  \right)
}|\lambda\left(  \left\{  F\left(  x_{j}\right)  u_{j}\left(  x\right)
\right\}  \right)  -F\left(  x\right)  \lambda\left(  \left\{  u_{j}\left(
x\right)  \right\}  \right)  |\\
& \\
&  =\frac{1}{\lambda\left(  \left\{  u_{j}\left(  x\right)  \right\}  \right)
}|\lambda\left(  \left\{  F\left(  x_{j}\right)  u_{j}\left(  x\right)
\right\}  \right)  -\lambda\left(  \left\{  F\left(  x\right)  u_{j}\left(
x\right)  \right\}  \right)  |\\
& \\
&  \leq\frac{1}{\lambda\left(  \left\{  u_{j}\left(  x\right)  \right\}
\right)  }\lambda\left(  \left\{  F(x_{j})u_{j}\left(  x\right)  \right\}
-\{F\left(  x\right)  u_{j}\left(  x\right)  \}\right) \\
& \\
&  =\frac{1}{\lambda\left(  \left\{  u_{j}\left(  x\right)  \right\}  \right)
}\lambda\left(  \{u_{j}\left(  x\right)  \left(  F\left(  x_{j}\right)
-F\left(  x\right)  \right)  \}\right)  \text{.}%
\end{align*}
Now by $\left(  3.6\right)  $ and the comment following%
\begin{align*}
\lambda\left(  \{u_{j}\left(  x\right)  \left(  F\left(  x_{j}\right)
-F\left(  x\right)  \right)  \}\right)   &  \leq2\left\Vert \{u_{j}\left(
x\right)  \left(  F\left(  x_{j}\right)  -F\left(  x\right)  \right)
\}\right\Vert _{\infty}\\
& \\
&  =2\max_{j}\left\{  u_{j}\left(  x\right)  \left\vert F\left(  x_{j}\right)
-F\left(  x\right)  \right\vert \right\}  \text{.}%
\end{align*}
Set $J=\{j:x\in C_{j}^{3}\}$. For $j\in J$ we have $\left\Vert x-x_{j}%
\right\Vert _{X}<\delta$ and so%
\[
u_{j}\left(  x\right)  \left\vert F\left(  x_{j}\right)  -F\left(  x\right)
\right\vert <u_{j}\left(  x\right)  \varepsilon.
\]
It follows that for $j\in J$%
\[
\frac{u_{j}\left(  x\right)  \left\vert F\left(  x_{j}\right)  -F\left(
x\right)  \right\vert }{\lambda\left(  \left\{  u_{j}\left(  x\right)
\right\}  \right)  }<\frac{u_{j}\left(  x\right)  \frac{\varepsilon}{2}%
}{\left\Vert \left\{  u_{j}\left(  x\right)  \right\}  \right\Vert _{\infty}%
}\leq\varepsilon.
\]
On the other hand, for $j\notin J$ we have by part (i) of Lemma 3, as noted
above, that%
\[
u_{j}\left(  x\right)  \left\vert F\left(  x_{j}\right)  -F\left(  x\right)
\right\vert \leq2u_{j}\left(  x\right)  <2\cdot\frac{\varepsilon}%
{4}=\varepsilon/2\text{.}%
\]
Hence, given that $\lambda\left(  \left\{  u_{j}\left(  x\right)  \right\}
\right)  \geq4/5$, we have for $j\notin J$%
\[
\frac{u_{j}\left(  x\right)  \left\vert F\left(  x_{j}\right)  -F\left(
x\right)  \right\vert }{\lambda\left(  \left\{  u_{j}\left(  x\right)
\right\}  \right)  }\leq\left(  \frac{5}{4}\right)  \left(  \frac{\varepsilon
}{2}\right)  <\varepsilon\text{.}%
\]
It follows that
\[
\left\vert K(x)-F(x)\right\vert <\varepsilon\text{.}%
\]

\end{proof}

\end{document}